\documentclass{article}

\usepackage{graphicx}
\usepackage{amscd}
\usepackage{graphics,amsmath,amssymb}
\usepackage{amsthm}
\usepackage{amsfonts,enumerate}
\usepackage{latexsym}
\usepackage{mathrsfs}

\newtheorem{theorem}{Theorem}
\theoremstyle{plain}

\newtheorem{corollary}{Corollary}

\newtheorem{definition}{Definition}

\newtheorem{lemma}{Lemma}

\numberwithin{equation}{section}

\begin{document}

%%%%%%%%%%%%%%%%%%%%%%%%%%%%%%%%%%%%%%%%%%%%%%%%%%%%%%%%%%%%%%%%
%%%%%%%%%%%%%%             T I T L E              %%%%%%%%%%%%%%
%%%%%%%%%%%%%%           A U T H O R S            %%%%%%%%%%%%%%
%%%%%%%%%%%%%%           A D D R E S S            %%%%%%%%%%%%%%
%%%%%%%%%%%%%%%%%%%%%%%%%%%%%%%%%%%%%%%%%%%%%%%%%%%%%%%%%%%%%%%%
\begin{center}
%\vskip 2cm
{\LARGE \bf
On the Harmonic and Hyperharmonic \medskip Fibonacci Numbers}\\
\vskip 1cm

{\large Naim TUGLU$^{1}$, Can KIZILATE\c{S}$^{2}$ and Seyhun KESIM$^{2}$} \\
\vskip 0.5cm

$^{1}$Gazi University, Faculty of Art and Science, Department of Mathematics,\smallskip \\
Teknikokullar 06500, Ankara-Turkey \smallskip \\
{\tt naimtuglu@gazi.edu.tr}\\
\vskip 0.5cm

$^{2}$B\"{u}lent Ecevit University, Faculty of Art and Science, Department of Mathematics, \smallskip\\ 67100, Zonguldak-Turkey \smallskip \\
{\tt cankizilates@gmail.com} \ \ \ {\tt seyhun.kesim@beun.edu.tr}
\vskip 0.8cm
\end{center}

%%%%%%%%%%%%%%%%%%%%%%%%%%%%%%%%%%%%%%%%%%%%%%%%%%%%%%%%%%%%%%%%
%%%%%%%%%%%%%%                                    %%%%%%%%%%%%%%
%%%%%%%%%%%%%%           A B S T R A C T          %%%%%%%%%%%%%%
%%%%%%%%%%%%%%                                    %%%%%%%%%%%%%%
%%%%%%%%%%%%%%%%%%%%%%%%%%%%%%%%%%%%%%%%%%%%%%%%%%%%%%%%%%%%%%%%

\begin{abstract}
In this paper, we study the theory of the harmonic and the hyperharmonic Fibonacci numbers. Also, we get some combinatoric identities like as harmonic and hyperharmonic numbers and we obtain some useful formulas for $\mathbb{F}_{n}$, which is finite sums of reciprocals of Fibonacci numbers.  We obtain spectral and Euclidean norms of circulant matrices involving harmonic and hyperharmonic Fibonacci numbers.\smallskip   \smallskip \\
\textbf{Mathematics Subject Classiffication:} 11B39, 05A19; 05A10, 15A60
\end{abstract}

%%%%%%%%%%%%%%%%%%%%%%%%%%%%%%%%%%%%%%%%%%%%%%%%%%%%%%%%%%%%%%%%
%%%%%%%%%%%%%%                                    %%%%%%%%%%%%%%
%%%%%%%%%%%%%%      I N T R O D U C T I O N       %%%%%%%%%%%%%%
%%%%%%%%%%%%%%                                    %%%%%%%%%%%%%%
%%%%%%%%%%%%%%%%%%%%%%%%%%%%%%%%%%%%%%%%%%%%%%%%%%%%%%%%%%%%%%%%
\thispagestyle{empty}

\section{Introduction}

The harmonic numbers have many applications in combinatorics and other areas. Many authors have studied these numbers. The $n^{\text{th}}$ harmonic number, denoted by $H_{n}$, is defined by
	\begin{equation*}
		H_{n}=\sum\limits_{k=1}^{n}\frac{1}{k}\text{,}
	\end{equation*}
where $H_{0}=0$. The $n^{\text{th}}$ harmonic number $H_{n}$ can be expressed as
	\begin{equation*}
		H_{n}=\dfrac{ {n+1 \brack 2} }{n!}
	\end{equation*}
where ${n \brack k}$ denotes the Stirling number of the first kind, counting the permutations of $n$ elements that are the product of $k$ disjoint cycles.

Several interesting properties of harmonic numbers can be found in \cite{1}.  For $n\geq 1$, some properties are following:
	\begin{eqnarray*}
		&&\sum\limits_{k=1}^{n-1}H_{k} =nH_{n}-n\text{,} \\
		&&\sum\limits_{k=1}^{n-1} \biggl(\displaystyle{k\atop m}\biggr) H_{k} = \biggl(\displaystyle{n\atop m+1}\biggr) \left (H_{n}-\frac{1}{m+1} \right )
		\text{,} \\
		&&\sum\limits_{k=1}^{n-1}kH_{k} =\frac{n^{\underline{2}}}{2}\left (H_{n}-\frac{1}{2} \right )
		\text{,}
	\end{eqnarray*}
where $n^{\underline{2}}=n(n-1)$ and $m$ is non-negative integer. 

In \cite{2}, Spivey proved that for $n\geq 1$,
	\begin{eqnarray*}
		&&\sum\limits_{k=0}^{n} \biggl(\displaystyle{n\atop k}\biggr) H_{k} =2^{n}\left (H_{n}-\sum\limits_{k=1}^{n}\dfrac{1}{k2^{k}}\right)\text{,} \\
		&&\sum\limits_{k=0}^{n} \biggl(\displaystyle{n\atop k}\biggr) (-1)^{k}H_{k} =-\dfrac{1}{n}\text{.}
	\end{eqnarray*}

Harmonic numbers have been generalized by many authors \cite{3,4,5,6,7}. In \cite{5}, Conway and Guy defined the $n^{\text{th}}$ hyperharmonic number of order $r$, $H_{n}^{(r)}$ for $n,r\geq 1$ by the following recurrence relations:
	\begin{equation*}
		H_{n}^{(r)}=\sum\limits_{k=1}^{n}H_{k}^{(r-1)}\text{,}
	\end{equation*}
where $H_{n}^{(0)}=\dfrac{1}{n}$ and $H_{n}^{(1)}=H_{n}$ is the $n^{\text{th}} $ ordinary harmonic number.

Furthermore these numbers can be expressed by binomial coefficients and ordinary harmonic numbers \cite{5}, as follows: 
	\begin{equation*}
		H_{n}^{(r)}=\binom{n+r-1}{r-1}(H_{n+r-1}-H_{r-1}).
	\end{equation*}

In \cite{4}, Benjamin and et all. gave the following properties of hyperharmonic numbers:
	\begin{eqnarray}
		&&H_{n}^{(r)}=\sum\limits_{t=1}^{n}\binom{n+r-t-1}{r-1}\frac{1}{t}\text{,} \label{har0} \\
		&&H_{n}^{(r)}=\sum\limits_{t=1}^{n}\binom{n+r-m-t-1}{r-m-1}H_{t}^{(m)},
	\end{eqnarray} 
where $0\leq m\leq r-1.$ In \cite{3}, Bahsi and Solak defined a special matrix whose entries are hyperharmonic numbers and gave some properties of this matrix.

The Fibonacci sequence is defined by the following recurrence relation, for $n\geq 0$
	\begin{equation*}
		F_{n+2}=F_{n+1}+F_{n}
	\end{equation*}
with $F_{0}=0$, $F_{1}=1.$ Similar to Fibonacci sequence, Lucas sequence is defined by the following recurrence relation, for $n\geq 0$
	\begin{equation*}
		L_{n+2}=L_{n+1}+L_{n}
	\end{equation*}
with $L_{0}=2$, $L_{1}=1.$

In \cite{8}, Dil and Mez\"{o} defined hyperfibonacci numbers, $F_{n}^{(r)}$, for $r$ positive integer 
	\begin{equation*}
		F_{n}^{(r)}=\sum_{k=0}^{n}F_{k}^{(r-1)}
	\end{equation*}
with $F_{n}^{(0)}=F_{n}$, $F_{0}^{(r)}=0$, $F_{1}^{(r)}=1$. Moreover authors obtained some properties of these numbers.

The Fibonacci zeta functions are defined by
	\begin{equation*}
		\zeta _{F}(s)=\sum_{k=1}^{\infty }\dfrac{1}{F_{k}^{s}}\text{,}
	\end{equation*}
and several aspects of the function have been studied, in \cite{9,10}. In 1989 Andr\'{e}-Jeannin \cite{12} proved that 
\begin{equation*}
\zeta _{F}(1)=\sum_{k=1}^{\infty }\frac{1}{F_{k}}=3,3598856662...
\end{equation*}
 is irrational.   

In \cite{10}, Ohtsuka and Nakamura studied the partial infinite sums of reciprocals Fibonacci numbers and the reciprocal of the square of the Fibonacci numbers. Holiday and Komatsu generalized their results in \cite{9}.

The important point to note here is the form of finite sums of reciprocals Fibonacci numbers. In \cite{13}, Rabinowitz pointed out that
	\begin{equation*}
		\mathbb{F}_{n}=\sum_{k=1}^{n}\frac{1}{F_{k}}
	\end{equation*}
is no known simple form. Motivated by the above papers, in this work, we investigate harmonic and hyperharmonic Fibonacci numbers and give some combinatorial properties of them. Now we present some preliminaries to our study.

In \cite{1}, difference operator of $f(x)$ is defined as 
	\begin{equation*}
		\Delta f(x)=f(x+1)-f(x).
	\end{equation*}

The expression $x$ to the $m$ falling is denoted $x^{\underline{m}}$. The value of
	\begin{equation*}
		x^{\underline{m}}=x(x-1)(x-2)\ldots (x-m+1)
	\end{equation*}
and it is called falling power.

$\Delta $ operator has very interesting property for $m\geq 0$
	\begin{equation*}
		\Delta x^{\underline{m}}=mx^{\underline{m-1}} \text{ .}
	\end{equation*}

Analogously, $\Delta $ operator has an inverse, the anti-difference or summation operator $\sum $ defined as follows. If $\Delta f(x)=g(x)$ then
	\begin{equation*}
		\sum\limits_{a}^{b}g(x)\delta _{x}=\sum\limits_{x=a}^{b-1}g(x)=f(b)-f(a).
	\end{equation*}

Anti-difference operator $\sum $ has some properties as follows
	\begin{equation*}
		\sum x^{\underline{m}}\delta _{x}=\left\{ 
			\begin{array}{cc}
				\dfrac{\,x^{\underline{m+1}}}{m+1\,} & \ \text{if }m\neq -1, \\ 
				&  \\ 
				H_{x} & \ \text{if }m=-1,
			\end{array}
		\right. 
	\end{equation*}
and 
	\begin{equation}
		\sum_{a}^{b}u(x)\Delta v(x)\delta _{x}=\left. u(x)v(x)\right\vert_{a}^{b+1}-\sum\limits_{a}^{b}Ev(x)\Delta u(x)\delta _{x}  \label{0}
	\end{equation}
where $Ev(x)=v(x+1)$ \cite{1}.

In this study, we investigate finite sums of reciprocals of Fibonacci numbers called harmonic Fibonacci number. Then we examine some combinatorics properties of harmonic Fibonacci numbers via difference operator. We define hyperharmonic Fibonacci number and investigate some properties same as hyperharmonic numbers. Finally we obtain norms of  some circulant matrices involving these numbers. 
%%%%%%%%%%%%%%%%%%%%%%%%%%%%%%%%%%%%%%%%%%%%%%%%%%%%%%%%%%%%%%%%%%%%%%%
%%%%%%%%%%%%%%                                           %%%%%%%%%%%%%%
%%%%%%%%%%%%%%      Harmonic Fibonacci Numbers and       %%%%%%%%%%%%%%
%%%%%%%%%%%%%%       Some Combinatorial Properties       %%%%%%%%%%%%%%
%%%%%%%%%%%%%%                                           %%%%%%%%%%%%%%
%%%%%%%%%%%%%%%%%%%%%%%%%%%%%%%%%%%%%%%%%%%%%%%%%%%%%%%%%%%%%%%%%%%%%%%
\section{Harmonic Fibonacci Numbers and Some Combinatorial Properties}

In this section, we investigate various properties of $\mathbb{F}_{n}$, finite sum of the reciprocals of the Fibonacci numbers are defined as
	\begin{equation*}
		\mathbb{F}_{n}=\sum\limits_{k=1}^{n}\frac{1}{F_{k}}\text{.}
	\end{equation*}
Here and subsequently we call it $n^{\text{th}}$ harmonic Fibonacci number. Now we state some theorems related to harmonic Fibonacci numbers.
\begin{theorem} \label{theorem21}
Let $\mathbb{F}_{n}$ be $n^{\text{th}}$ harmonic Fibonacci number. Then we have 
	\begin{equation*}
		\sum_{k=0}^{n-1}\mathbb{F}_{k}=n\mathbb{F}_{n}-\sum_{k=0}^{n-1}\frac{k+1}{F_{k+1}}\text{.}
	\end{equation*}
\end{theorem}
\begin{proof}
Our proof starts with the observation that property of anti-difference operator. Let $u(k)=\mathbb{F}_{k}$ and $\Delta v(k)=1$ be in (\ref{0}). Then, we obtain $\Delta u(k)=\frac{1}{F_{k+1}}$, $v(k)=k$ \ and $Ev(k)=k+1.$ Hence, we have 
	\begin{equation*}
		\sum_{k=0}^{n-1}\mathbb{F}_{k}=n\mathbb{F}_{n}-\sum_{k=0}^{n-1}\frac{k+1}{F_{k+1}}\text{.}
	\end{equation*}
\end{proof}
\begin{theorem}
Let $\mathbb{F}_{n}$ be $n^{\text{th}}$ harmonic Fibonacci number. Then we have 
\begin{equation}
\sum_{k=0}^{n-1} \mathbb{F}_{k}^{2}=n\mathbb{F}_{n}^{2}-\sum\limits_{k=0}^{n-1}\frac{k+1}{F_{k+1}}\left(2\mathbb{F}_{k}+\frac{1}{F_{k+1}}\right) \text{.}  \label{norm1}
\end{equation}
\end{theorem}
\begin{proof}
Let $u(k)=\mathbb{F}_{k}^{2}$ and $\Delta v(k)=1$ be in (\ref{0}). Then by
using the equation (\ref{0}), we obtain 
\begin{equation*}
\sum_{k=0}^{n-1} \mathbb{F}_{k} ^{2}=n\mathbb{F}_{n}^{2}-\sum\limits_{k=0}^{n-1}\frac{k+1}{F_{k+1}}\left(2\mathbb{F}_{k}+\frac{1}{F_{k+1}}\right) \text{.}
\end{equation*}
\end{proof}
\begin{theorem}
Let $\mathbb{F}_{n}$ be $n^{\text{th}}$ harmonic Fibonacci number and $m$ is a nonnegative integer
	\begin{equation*}
		\sum_{k=0}^{n-1}\binom{k}{m}\mathbb{F}_{k}=\binom{n}{m+1}\mathbb{F}_{n}-\sum_{k=0}^{n-1}\binom{k+1}{m+1}\frac{1}{F_{k+1}}.
	\end{equation*}
\end{theorem}
\begin{proof}
Let $u(k)=\mathbb{F}_{k}$ and $\Delta v(k)=\binom{k}{m}$ be in (\ref{0}). Then, we obtain $\Delta u(k)=\frac{1}{F_{k+1}}$, $v(k)=\binom{k}{m+1}$ and $ Ev(k)=\binom{k+1}{m+1}.$ By using the equation (\ref{0}), we have
	\begin{equation*}
		\sum_{k=0}^{n-1}\binom{k}{m}\mathbb{F}_{k}=\binom{n}{m+1}\mathbb{F}_{n}-\sum_{k=0}^{n-1}\binom{k+1}{m+1}\frac{1}{F_{k+1}}.
	\end{equation*}
\end{proof}
\begin{theorem}
Let $\mathbb{F}_{n}$ be $n^{\text{th}}$ harmonic Fibonacci number. Then we have 
	\begin{equation*}
		\sum_{k=0}^{n-1}k^{\underline{m}}\,\mathbb{F}_{k}=\frac{n^{\underline{m+1}}}{m+1}\,\mathbb{F}_{n}-\sum_{k=0}^{n-1}\frac{(k+1)^{\underline{m+1}}}{m+1}\frac{1}{F_{k+1}}.
	\end{equation*}
\end{theorem}
\begin{proof}
Let $u(k)=\mathbb{F}_{k}$ and $\Delta v(k)=k^{\underline{m}}$ in (\ref{0}). Then, we get $\Delta u(k)=\dfrac{1}{F_{k+1}}$, \ \ $v(k)=\dfrac{k^{\underline{m+1}}}{m+1}$ and $Ev(k)=\dfrac{(k+1)^{\underline{m+1}}}{m+1}.$ By aid of the equation (\ref{0}), we have 
	\begin{equation*}
		\sum_{k=0}^{n-1}k^{\underline{m}}\,\mathbb{F}_{k}=\frac{n^{\underline{m+1}}}{m+1}\,\mathbb{F}_{n}-\sum_{k=0}^{n-1}\frac{(k+1)^{\underline{m+1}}}{m+1}\frac{1}{F_{k+1}}.
	\end{equation*}
\end{proof}

The following theorem gives the relationship between harmonic numbers and harmonic Fibonacci numbers.
\begin{theorem}\label{Theorem2}
Let $\mathbb{F}_{n}$ be $n^{\text{th}}$ harmonic Fibonacci number, then
	\begin{equation*}
		\sum_{k=0}^{n-1}\frac{\mathbb{F}_{k}}{k+1}=H_{n}\mathbb{F}_{n}-\sum_{k=0}^{n-1}\frac{H_{k+1}}{F_{k+1}}.
	\end{equation*}
\end{theorem}
\begin{proof}
Let $u(k)=\mathbb{F}_{k}$ and $\Delta v(k)=\dfrac{1}{k+1}$ in (\ref{0}). Then, we obtain $\Delta u(k)=\dfrac{1}{F_{k+1}}$, $v(k)=H_{k}$ and $Ev(k)=H_{k+1}.$ By using the equation (\ref{0}), we have 
	\begin{equation*}
		\sum_{k=0}^{n-1}\frac{\mathbb{F}_{k}}{k+1}=H_{n}\mathbb{F}_{n}-\sum_{k=0}^{n-1}\frac{H_{k+1}}{F_{k+1}}.
	\end{equation*}
\end{proof}
\begin{corollary}
	\begin{equation*}
		\sum_{k=0}^{n-1}\dfrac{\mathbb{F}_{k}}{k+1}=\dfrac{{n+1 \brack 2}}{n!}\mathbb{F}_{n}-\sum_{k=0}^{n-1}\dfrac{{k+2 \brack 2}}{(k+1)!F_{k+1}} \text{.}
	\end{equation*}
\end{corollary}
\begin{proof}
If we take $H_{n}=\dfrac{{n+1 \brack 2}}{n!}$ in Theorem \ref{Theorem2} the proof can be completed. 
\end{proof}
\begin{theorem}
We have 
	\begin{equation*}
		\sum_{k=0}^{n-1}F_{k-1}\mathbb{F}_{k}=F_{n}\mathbb{F}_{n}-n.
	\end{equation*}
\end{theorem}
\begin{proof}
Let $u(k)=\mathbb{F}_{k}$ and $\Delta v(k)=F_{k-1}$ in (\ref{0}). Then we obtain $\Delta u(k)=\dfrac{1}{F_{k+1}}$, $v(k)=F_{k}$ and $Ev(k)=F_{k+1}.$ By using the equation (\ref{0}), we have 
	\begin{eqnarray*}
		\sum_{k=0}^{n-1}F_{k-1}\mathbb{F}_{k} 
		&=& F_{n}\mathbb{F}_{n}-\sum_{k=0}^{n-1}k^{\underline{0}}\:\delta _{k} \\ 
		&=& F_{n}\mathbb{F}_{n}-n.
	\end{eqnarray*}
\end{proof}
\begin{theorem}
	Let $\mathbb{F}_{n}$ be $n^{\text{th}}$ harmonic Fibonacci number and $L_{n}$ is Lucas number. Then we have 
		\begin{equation*}
			\sum_{k=0}^{n-1}L_{k-1}\mathbb{F}_{k}=L_{n}\mathbb{F}_{n}-\sum_{k=0}^{n-1}\frac{L_{k+1}}{F_{k+1}}.
		\end{equation*}
\end{theorem}
\begin{proof}
	The proof may be handled in much the same way.
\end{proof}
%%%%%%%%%%%%%%%%%%%%%%%%%%%%%%%%%%%%%%%%%%%%%%%%%%%%%%%%%%%%%%%%%%%%%%%
%%%%%%%%%%%%%%                                           %%%%%%%%%%%%%%
%%%%%%%%%%%%%%      Hyperharmonic Fibonacci Numbers      %%%%%%%%%%%%%%
%%%%%%%%%%%%%%                                           %%%%%%%%%%%%%%
%%%%%%%%%%%%%%%%%%%%%%%%%%%%%%%%%%%%%%%%%%%%%%%%%%%%%%%%%%%%%%%%%%%%%%%
\section{Hyperharmonic Fibonacci Numbers}

In this section, hyperharmonic Fibonacci numbers will be defined having used a similar of affiliation between harmonic numbers and hyperharmonic numbers by the help of harmonic Fibonacci numbers. Now we define $\mathbb{F} _{n}^{(r)}$, the hyperharmonic Fibonacci numbers of order $r$.

\begin{definition}
	Let $\mathbb{F}_{n}$ be $n^{\text{th}}$ harmonic Fibonacci number. For $n,r\geq 1$ Hyperharmonic Fibonacci numbers are defined by
		\begin{equation}	 
			\mathbb{F}_{n}^{(r)}=\sum\limits_{k=1}^{n}\mathbb{F}_{k}^{(r-1)}  \label{1}
		\end{equation}
	with $\mathbb{F}_{n}^{(0)}=\dfrac{1}{F_{n}}$ and $\mathbb{F}_{0}^{(k)}=0$ for $k\geq 0$.
\end{definition}

In particular for $r=1$ we get
	\begin{equation}
		\mathbb{F}_{n}^{(1)}=\mathbb{F}_{n}=\sum\limits_{k=1}^{n}\frac{1}{F_{k}} \label{2}
	\end{equation}
where $\mathbb{F}_{n}$ is $n^{\text{th}}$ harmonic Fibonacci numbers. 

\begin{center}
%\begin{table}[h!]
		\begin{tabular}{l|l c c c c}
			\hline
				$n$ & $1$ & $2$ & $3$ & $4$ & $5$ \\ \hline \smallskip
				$\mathbb{F}_{n}^{(1)}$ & $1$ & $2$ & $\frac{5}{2}$ & $\frac{17}{6}$ & $\frac{91}{30}$ \smallskip \\ 
				$\mathbb{F}_{n}^{(2)}$ & $1$ & $3$ & $\frac{11}{2}$ & $\frac{25}{3}$ & $\frac{341}{30}$ \smallskip  \\ 
				$\mathbb{F}_{n}^{(3)}$ & $1$ & $4$ & $\frac{19}{2}$ & $\frac{107}{6}$ & $\frac{146}{5}$ \smallskip \\ 
				$\mathbb{F}_{n}^{(4)}$ & $1$ & $5$ & $\frac{29}{2}$ & $\frac{97}{3}$ & $\frac{923}{15}$ \smallskip \\ \hline
		\end{tabular} \medskip \\
		\textbf{Table 1. Some Hyperharmonic Fibonacci numbers.}
%	\end{table}
\end{center}
\begin{lemma}
	Hyperharmonic Fibonacci numbers have the recurrence relation as follows:
		\begin{equation*}
			\mathbb{F}_{n}^{(r)}=\mathbb{F}_{n}^{(r-1)}+\mathbb{F}_{n-1}^{(r)}\text{.}
		\end{equation*}
\end{lemma}
\begin{proof}
	From equation (\ref{1}) we have 
		\begin{eqnarray*}
			\mathbb{F}_{n}^{(r)} &=&\sum\limits_{k=1}^{n}\mathbb{F}_{k}^{(r-1)} \\
			&=&\sum\limits_{k=1}^{n-1}\mathbb{F}_{k}^{(r-1)}+\mathbb{F}_{n}^{(r-1)} \\
			&=&\mathbb{F}_{n-1}^{(r)}+\mathbb{F}_{n}^{(r-1)}.
		\end{eqnarray*}
\end{proof}
We give an interesting property of hyperharmonic numbers same as in (\ref{har0}).
\begin{theorem}\label{Theorem3}
	For $1\leq i,j\leq n$, we have
		\begin{equation*}
			\mathbb{F}_{n-i+1}^{(j)}=\sum_{k=i}^{n}\binom{n-k+j-1}{j-1} \frac{1}{F_{k-i+1}}.
		\end{equation*}
\end{theorem}
\begin{proof}
We begin by recalling the definition of $\mathbb{F}_{n}^{(r)}$. If we use this definition $j-1$ times, we get
	\begin{eqnarray*}
		\mathbb{F}_{n-i+1}^{(j)} &=&\sum\limits_{k=1}^{n-i+1}\mathbb{F}_{k}^{(j-1)}\\ &=&\sum_{k_{j}=1}^{n-i+1}\sum_{k_{j-1}=1}^{k_{j}}\ldots\sum_{k_{1}=1}^{k_{2}}\frac{1}{F_{k_{1}}}\text{.}
	\end{eqnarray*}
We use induction on $n$ to obtain
	\begin{equation*}
		\sum_{k_{j}=1}^{n-i+1}\sum_{k_{j-1}=1}^{k_{j}}\ldots \sum_{k_{1}=1}^{k_{2}} \frac{1}{F_{k_{1}}}=\sum_{k=i}^{n}\binom{n-k+j-1}{j-1}\frac{1}{F_{k-i+1}}.
	\end{equation*}
Clearly it is true for $n=1$. Suppose it is true for some $n>1$, then using the induction hypothesis, we have
	\begin{eqnarray*}
		\sum_{k_{j}=1}^{n-i+2}\sum_{k_{j-1}=1}^{k_{j}}\ldots \sum_{k_{1}=1}^{k_{2}} \frac{1}{F_{k_{1}}} &=&\sum_{k_{j}=1}^{n-i+1}\sum_{k_{j-1}=1}^{k_{j}}\ldots \sum_{k_{1}=1}^{k_{2}}\frac{1}{F_{k_{1}}}+\sum_{k_{j-1}=1}^{n-i+2} \sum_{k_{j-2}=1}^{k_{j-1}}\ldots \sum_{k_{1}=1}^{k_{2}}\frac{1}{F_{k_{1}}} \\ &=&\sum_{k_{j}=1}^{n-i+1}\sum_{k_{j-1}=1}^{k_{j}}\ldots \sum_{k_{1}=1}^{k_{2}}\frac{1}{F_{k_{1}}}+\sum_{k_{j-1}=1}^{n-i+1} \sum_{k_{j-2}=1}^{k_{j-1}}\ldots \sum_{k_{1}=1}^{k_{2}}\frac{1}{F_{k_{1}}} \\ &&+\sum_{k_{j-2}=1}^{n-i+2}\sum_{k_{j-3}=1}^{k_{j-2}}\ldots \sum_{k_{1}=1}^{k_{2}}\frac{1}{F_{k_{1}}} \\ &=&\sum_{k_{j}=1}^{n-i+1}\sum_{k_{j-1}=1}^{k_{j}}\ldots \sum_{k_{1}=1}^{k_{2}}\frac{1}{F_{k_{1}}}+\sum_{k_{j-1}=1}^{n-i+1} \sum_{k_{j-2}=1}^{k_{j-1}}\ldots \sum_{k_{1}=1}^{k_{2}}\frac{1}{F_{k_{1}}} +\cdots  \\ &&+\sum\limits_{k_{2}=0}^{n-i+1}\sum\limits_{k_{1}=0}^{k_{2}}\frac{1}{F_{k_{1}}}+\sum\limits_{k_{1}=0}^{n-i+2}\frac{1}{F_{k_{1}}} \\ &=&\sum_{k_{j}=1}^{n-i+1}\sum_{k_{j-1}=1}^{k_{j}}\ldots \sum_{k_{1}=1}^{k_{2}}\frac{1}{F_{k_{1}}}+\sum_{k_{j-1}=1}^{n-i+1} \sum_{k_{j-2}=1}^{k_{j-1}}\ldots \sum_{k_{1}=1}^{k_{2}}\frac{1}{F_{k_{1}}} +\cdots  \\ &&+\sum\limits_{k_{2}=0}^{n-i+1}\sum\limits_{k_{1}=0}^{k_{2}}\frac{1}{F_{k_{1}}}+\sum\limits_{k_{1}=0}^{n-i+1}\frac{1}{F_{k_{1}}}+\frac{1}{F_{n-i+2}} \\ &=&\sum_{k=i}^{n}\frac{1}{F_{k-i+1}}\left[ \binom{n-k+j-1}{j-1}+\binom{n-k+j-2}{j-2}+\cdots +\binom{n-k}{0}\right]  \\ &&+\frac{1}{F_{n-i+2}} \\ &=&\sum_{k=i}^{n}\frac{1}{F_{k-i+1}}\binom{n-k+j}{j-1}+\frac{1}{F_{n-i+2}} \\ &=&\sum_{k=i}^{n+1}\binom{n-k+j}{j-1}\frac{1}{F_{k-i+1}}.
	\end{eqnarray*}
Finally we obtain
	\begin{equation*}
		\mathbb{F}_{n-i+1}^{(j)}=\sum_{k=i}^{n}\binom{n-k+j-1}{j-1}\frac{1}{F_{k-i+1}}.
	\end{equation*}
\end{proof}
\begin{corollary}
We have 
	\begin{equation*}
		\mathbb{F}_{n}^{(r)}=\sum\limits_{k=1}^{n}\binom{n-k+r-1}{r-1}\frac{1}{F_{k}}\text{.}
	\end{equation*}
\end{corollary}
\begin{proof}
By putting $i=1$ and $j=r$ in Theorem \ref{Theorem3}, one has
	\begin{equation*}
		\mathbb{F}_{n}^{(r)}=\sum\limits_{k=1}^{n}\binom{n-k+r-1}{r-1}\frac{1}{F_{k}}\text{.}
	\end{equation*}
\end{proof}

At this point, we express $\mathbb{F}_{n}^{(r+s)}$ in terms of $\mathbb{F}_{1}^{(s)},\mathbb{F}_{2}^{(s)},\ldots ,\mathbb{F}_{n}^{(s)}$ with following the theorem.
\begin{theorem}
	For $r\geq 1$ and $s\geq 0$ we have
		\begin{equation}
			\mathbb{F}_{n}^{(r+s)}=\sum\limits_{t=1}^{n}\binom{n-t+r-1}{r-1}\mathbb{F}_{t}^{(s)}  \label{h}
		\end{equation}
\end{theorem}
\begin{proof}
We prove this by induction on $n$. Clearly it is true for $n=1$. Assuming (\ref{h}) to hold for $n>1$, we will prove it for $n+1$. Thus,
	\begin{eqnarray*}
		\mathbb{F}_{n+1}^{(r+s)} &=&\mathbb{F}_{n+1}^{(r+s-1)}+\mathbb{F}_{n}^{(r+s)} \\ &=&\mathbb{F}_{n+1}^{(r+s-2)}+\mathbb{F}_{n}^{(r+s-1)}+\mathbb{F}_{n}^{(r+s)} \\ &&\vdots  \\
		&=&\mathbb{F}_{n+1}^{(s)}+\mathbb{F}_{n}^{(s+1)}+\cdots +\mathbb{F}_{n}^{(r+s-1)}+\mathbb{F}_{n}^{(r+s)} \\ &=&\mathbb{F}_{n+1}^{(s)}+\sum\limits_{t=1}^{n}\left[ \binom{n-t}{0}+\binom{n-t+1}{1}+\cdots +\binom{n-r-t-1}{r-1}\right] \mathbb{F}_{t}^{(s)} \\ &=&\mathbb{F}_{n+1}^{(s)}+\sum\limits_{t=1}^{n}\binom{n-t+r}{r-1}\mathbb{F}_{t}^{(s)} \\
		&=&\sum\limits_{t=1}^{n+1}\binom{n-t+r}{r-1}\mathbb{F}_{t}^{(s)}.
	\end{eqnarray*}
Thus the proof is completed by the mathematical induction.
\end{proof}
\begin{proof}[Another Proof.]
Let $C_{n}^{(r)}$ be $n\times n$ matrix which is defined by
	\begin{equation*}
		C_{n}^{(r)}=
			\begin{pmatrix}
				\mathbb{F}_{n}^{(r)} & \mathbb{F}_{n}^{(r+1)} & ... & \mathbb{F}_{n}^{(r+n-1)} \\ 
				\mathbb{F}_{n-1}^{(r)} & \mathbb{F}_{n-1}^{(r+1)} & ... & \mathbb{F}_{n-1}^{(r+n-1)} \\ 
				\vdots  & \vdots  & \cdots  & \vdots  \\ 
				\mathbb{F}_{1}^{(r)} & \mathbb{F}_{1}^{(r+1)} & ... & \mathbb{F}_{1}^{(r+n-1)}
			\end{pmatrix}
		,
	\end{equation*}
where $\mathbb{F}_{n}^{(r)}$ is the $n^{\text{th}}$ hyperharmonic Fibonacci number.
Let
	\begin{equation*}
		A=
			\begin{pmatrix}
				1 & 1 & \cdots  & 1 \\ 
				0 & 1 & \cdots  & 1 \\ 
				\vdots  & \vdots  & \ddots  & \vdots  \\ 
				0 & 0 & \cdots  & 1
			\end{pmatrix}
		,
	\end{equation*}
be $n\times n$ upper triangle matrix. In \cite{3}, Bahsi and Solak show that $A^{r}=(b_{ij})_{n\times n},$ where
	\begin{equation*}
		b_{ij}=\left\{ 
			\begin{array}{cc} 
				\binom{j-i+r-1}{r-1} & \text{if }i\leq j, \\ 
				0 & \text{otherwise.}
			\end{array}
		\right. 
	\end{equation*}
From the matrix multiplication it is obtained 
	\begin{equation*}
		C_{n}^{(r+s)}=A^{r}C_{n}^{(s)}\text{.}
	\end{equation*}
Hence the element of $\left( C_{n}^{(r+s)}\right) _{11}$ is
	\begin{eqnarray*}
		\mathbb{F}_{n}^{(r+s)} &=&\sum\limits_{j=1}^{n}b_{1j}\mathbb{F}_{n-j+1}^{(s)} \\
		&=&\sum\limits_{j=1}^{n}\binom{j+r-2}{r-1}\mathbb{F}_{n-j+1}^{(s)} \\ &=&\sum\limits_{t=1}^{n}\binom{n-t+r-1}{r-1}\mathbb{F}_{t}^{(s)}\text{.}
	\end{eqnarray*}
\end{proof}

%%%%%%%%%%%%%%%%%%%%%%%%%%%%%%%%%%%%%%%%%%%%%%%%%%%%%%%%%%%%%%%%%%%%%%%%%%%%%
%%%%%%%%%%%%%%                                                 %%%%%%%%%%%%%%
%%%%%%%%%%%%%%  An Application of Harmonic an Hyperharmonic    %%%%%%%%%%%%%%
%%%%%%%%%%%%%%   Fibonacci Numbers in Circulant Matrices       %%%%%%%%%%%%%%
%%%%%%%%%%%%%%                                                 %%%%%%%%%%%%%%
%%%%%%%%%%%%%%%%%%%%%%%%%%%%%%%%%%%%%%%%%%%%%%%%%%%%%%%%%%%%%%%%%%%%%%%%%%%%%
\section{An Application of Harmonic and Hyperharmonic Fibonacci Numbers in Circulant Matrices}

In this section, we will give some application on matrix norms of harmonic Fibonacci and hyperharmonic Fibonacci numbers. Recently, there have been many papers on the norms of circulant matrices with special numbers \cite{14,15}.

Let $A=(a_{ij})$ be any $m\times n$ complex matrix. The Euclidean norm and spectral norm of the matrix $A$ are respectively,
	\begin{equation*}
		\left\Vert A\right\Vert _{E}=\left(\sum\limits_{i=1}^{m}\sum\limits_{j=1}^{n}\left\vert a_{ij}\right\vert^{2}\right) ^{\frac{1}{2}}
	\end{equation*}
and
	\begin{equation*}
	\left\Vert A\right\Vert _{2}=\sqrt{\max_{1\leq i\leq n}\left\vert \lambda_{i}(A^{H}A)\right\vert }
	\end{equation*}
where $\lambda _{i}(A^{H}A)$ is eigenvalue of $A^{H}A$ and $A^{H}$ is conjugate transpose of the matrix $A$. Then the following inequality holds:
	\begin{equation}
		\left\Vert A\right\Vert _{2}\leq \left\Vert A\right\Vert _{E}\leq \sqrt{n}\left\Vert A\right\Vert _{2}\text{.}  \label{3}
	\end{equation}
A circulant matrix of order $n$ is meant a square matrix of the form
	\begin{equation*}
		C=Circ(c_{0},c_{1},\ldots ,c_{n-1})=
			\begin{pmatrix}
				c_{0} & c_{1} & c_{2} & \ldots  & c_{n-1} \\ 
				c_{n-1} & c_{0} & c_{1} & \ldots  & c_{n-2} \\ 
				\vdots  & \vdots  & \vdots  &  & \vdots  \\ 
				c_{1} & c_{2} & c_{3} & \ldots  & c_{0}%
			\end{pmatrix}
	\end{equation*}
\begin{theorem} 
Let $C_{1}=Circ\,\bigl (\mathbb{F}_{0},\mathbb{F}_{1},\mathbb{F}_{2},\ldots ,\mathbb{F}_{n-1}\bigr )$ be $n\times n$ circulant matrix. The Euclidean norm of $C_{1}$ is
	\begin{equation*}
		\left\Vert C_{1}\right\Vert _{E}=\left[ n^{2}\mathbb{F}_{n}^{2}-n\sum\limits_{k=0}^{n-1}\frac{k+1}{F_{k+1}}\left( 2\mathbb{F}_{k}+\frac{1}{F_{k+1}}\right) \right] ^{\frac{1}{2}}.
	\end{equation*}
\end{theorem}
\begin{proof}
From the definition of Euclidean norm, we have
	\begin{equation*}
		\left\Vert C_{1}\right\Vert _{E}^{2}=n\sum_{k=0}^{n-1}\mathbb{F}_{k}^{2}\text{.}
	\end{equation*}
Then by using the equation (\ref{norm1}), we have
	\begin{equation*}
		\left\Vert C_{1}\right\Vert _{E}=\left[ n^{2}\mathbb{F}_{n}^{2}-n\sum\limits_{k=0}^{n-1}\frac{k+1}{F_{k+1}}\left( 2\mathbb{F}_{k}+\frac{1}{F_{k+1}}\right) \right] ^{\frac{1}{2}}\text{.}
	\end{equation*}
\end{proof}

For the proof of the following theorem, we use same method in \cite{14}.
\begin{theorem} \label{theorem4}
Let $C_{1}=Circ\,\bigl(\mathbb{F}_{0},\mathbb{F}_{1},\mathbb{F}_{2},\ldots ,\mathbb{F}_{n-1}\bigl)$ be $n\times n$ circulant matrix. The spectral norm of $C_{1}$ is
	\begin{equation*}
		\left\Vert C_{1}\right\Vert _{2}=n\mathbb{F}_{n}-\sum_{k=0}^{n-1}\frac{k+1}{F_{k+1}}\text{.}
	\end{equation*}
\end{theorem}
\begin{proof}
Since the circulant matrices are normal, the spectral norm of the circulant $C_{1}$ is equal to its spectral radius. Furthermore, $C_{1}$ is irreducible and its entries are nonnegative, we have that the spectral radius of the matrix $C_{1}$ is equal to its Perron root. Let $\upsilon $ be a vector with
all components $1$. Then
	\begin{equation*}
		C_{1}\upsilon =\left( \sum_{k=0}^{n-1}\mathbb{F}_{k}\right) \upsilon \text{.}
	\end{equation*}
Obviously, $\sum\limits_{k=0}^{n-1}\mathbb{F}_{k}$ is an eigenvalues of $C_{1}$. Corresponding a positive eigenvector, it must be the Perron root of the matrix $C_{1}$. Hence from the Theorem \ref{theorem21}, we have
	\begin{equation*}
		\left\Vert C_{1}\right\Vert _{2}=n\mathbb{F}_{n}-\sum_{k=0}^{n-1}\frac{k+1}{F_{k+1}}\text{.}
	\end{equation*}
\end{proof}
\begin{theorem} \label{theorem43}
Let $C_{2}=Circ\,\bigl(\mathbb{F}_{0}^{(r)},\mathbb{F}_{1}^{(r)},\mathbb{F}_{2}^{(r)},\ldots ,\mathbb{F}_{n-1}^{(r)}\bigl)$ be $n\times n$ circulant matrix.
The spectral norm of $C_{2}$ is
	\begin{equation*}
		\left\Vert C_{2}\right\Vert _{2}=\mathbb{F}_{n-1}^{(r+1)}\text{.}
	\end{equation*}
\end{theorem}
\begin{proof}
Analysis similar to that in the proof of Theorem \ref{theorem4} shows that
	\begin{equation*}
		\left\Vert C_{2}\right\Vert _{2}=\sum_{k=0}^{n-1}\mathbb{F}_{k}^{(r)}\text{.}
	\end{equation*}
From the definition of hyperharmonic Fibonacci numbers, we have
	\begin{equation*}
		\left\Vert C_{2}\right\Vert _{2}=\mathbb{F}_{n-1}^{(r+1)}\text{.}
	\end{equation*}
\end{proof}
\begin{corollary} \label{corollary3}
For the Euclidean norm of the matrix $C_{2}=Circ \bigl(\mathbb{F}_{0}^{(r)},\mathbb{F}_{1}^{(r)},\mathbb{F}_{2}^{(r)},\ldots ,\mathbb{F}_{n-1}^{(r)}\bigl)$, we have
	\begin{equation*}
		\mathbb{F}_{n-1}^{(r+1)}\leq \left\Vert C_{2}\right\Vert _{E}\leq \sqrt{n} \:\,\mathbb{F}_{n-1}^{(r+1)}\text{.}
	\end{equation*}
\end{corollary}
\begin{proof}
The proof is trivial from Theorem \ref{theorem43} and the relation between spectral norm and Euclidean norm in (\ref{3}).
\end{proof}
\begin{corollary}
For the sum of the squares of hyperharmonic Fibonacci numbers, we have
	\begin{equation*}
		\frac{1}{\sqrt{n}}\:\mathbb{F}_{n-1}^{(r+1)}\leq \sqrt{\sum_{k=0}^{n-1}\left(\mathbb{F}_{k}^{(r)}\right) ^{2}}\leq \mathbb{F}_{n-1}^{(r+1)}\text{.}
	\end{equation*}
\end{corollary}
\begin{proof}
It is easy seen that from the definition of Euclidean norm and Corollary \ref{corollary3}
\end{proof}

%%%%%%%%%%%%%%%%%%%%%%%%%%%%%%%%%%%%%%%%%%%%%%%%%%%%%%%%%%%%
%%%%%%%%%%%%%%                                %%%%%%%%%%%%%%
%%%%%%%%%%%%%%            References          %%%%%%%%%%%%%%
%%%%%%%%%%%%%%                                %%%%%%%%%%%%%%
%%%%%%%%%%%%%%%%%%%%%%%%%%%%%%%%%%%%%%%%%%%%%%%%%%%%%%%%%%%%

\end{document}